\documentclass[12pt,bezier]{article}
\usepackage{times}
\usepackage{booktabs}
\usepackage{pifont}
\usepackage{floatrow}
\usepackage{caption}
\usepackage{mathrsfs}
\usepackage[fleqn]{amsmath}
\usepackage{amsfonts,amsthm,amssymb,mathrsfs,bbding}
\usepackage{txfonts}
\usepackage{graphics,multicol}
\usepackage{graphicx}
\usepackage{color}
\usepackage{caption}
\usepackage{cite}
\usepackage{latexsym,bm}
\usepackage{indentfirst}
\evensidemargin=\oddsidemargin\topmargin=-1.5cm

\newtheorem{prob}{Problem}
\newtheorem{lem}{Lemma}[section]
\newtheorem{thm}{Theorem}[section]

\newtheorem{conj}{Conjecture}
\newtheorem{defi}{Definition}[section]
\theoremstyle{definition}

\addtocounter{section}{0}

\begin{document}
\title{A spectral extremal problem on graphs with given size and matching number
\footnote{Supported by National Natural Science Foundation of China (No. 11971445) and the research project of Chuzhou University (No. 2020qd18).}}

\author{
{ Mingqing Zhai$^{a}$}
, { Jie Xue$^{b}$}\thanks{Corresponding author. E-mail addresses: mqzhai@chzu.edu.cn
(M. Zhai); jie\_xue@126.com (J. Xue); rfliu@zzu.edu.cn (R. Liu).}
, { Ruifang Liu$^{b}$}
\\
{\footnotesize $^a$ School of Mathematics and Finance, Chuzhou University, Chuzhou, Anhui 239012, China} \\
{\footnotesize $^b$ School of Mathematics and Statistics, Zhengzhou University, Zhengzhou, Henan 450001, China}
}

\date{}
\maketitle {\flushleft\large\bf Abstract}

Brualdi and Hoffman (1985) proposed the problem of determining the maximal spectral radius of graphs with given size.
In this paper, we consider the Brualdi-Hoffman type problem of graphs with given matching number. The maximal $Q$-spectral radius of graphs with given size and matching number is obtained, and the corresponding extremal graphs are also determined.

\vspace{0.1cm}
\begin{flushleft}
\end{flushleft}

\section{Introduction}\label{s-1}
Unless stated otherwise, we follow \cite{Boundy2008,Cvetkovic2010} for terminology and notations. All graphs considered here are simple and undirected. Let $G$ be a graph with vertex set $V(G)$ and edge set $E(G)$. The order of $G$ is the number of vertices in $V(G)$. The number of edges in $E(G)$ is called the size of $G$, and denoted by $m(G)$. The signless Laplacian matrix of $G$ is denoted by $Q(G)$.
The largest eigenvalue of $Q(G)$ is called the $Q$-spectral radius of $G$, and write $q(G)$.

Spectral extremal problem is a classical issue in graph spectra theory.
The core of the issue is to ask the extremal value of a spectral parameter of graphs under some constraints.
In 1985, Brualdi and Hoffman proposed the following spectral extremal problem:

\begin{prob}{(Brualdi-Hoffman Problem)}\label{B-H problem}
  For graphs of size $m$, what is the maximum of a spectral parameter?
\end{prob}

When $m$ is equal to $\binom{k}{2}$ for some integer $k$, Brualdi and Hoffman \cite{Brualdi1985} determined the maximal spectral radius of graphs of size $m$. Moreover, Brualdi and Hoffman proposed a conjecture for any $m$.

\begin{conj}{(Brualdi and Hoffman \cite{Brualdi1985})}
  If $m=\binom{k}{2}+s$, where $s<k$, the maximal spectral radius of a graph $G$ of size $m$ is attained by taking the complete graph $K_{k}$ with $k$ vertices and adding a new vertex which is joined to $s$ of the vertices of $K_{k}$.
\end{conj}

Friedland \cite{Friedland1985} confirmed the conjecture for some special cases. Finally, the conjecture was proved by Rowlinson \cite{Rowlinson1988}. However, the Brualdi-Hoffman problem will become more difficult if adding some constraints, such as the size and order are both fixed. Brualdi and Solheid \cite{Brualdi1986} considered the maximal spectral radius of connected graphs of order $n$ and size $m=n+k$, where $k\geq 0$. If $0\leq k\leq 5$ and $n$ is sufficiently large, it was proved that the extremal graph is the graph obtained from the star $K_{1,n-1}$ by adding edges from a pendant vertex to the other pendant vertices. Brualdi and Solheid conjectured that the above conclusion also holds for all $k$, when $n$ is sufficiently large with respect to $k$. Later, Cvetkovi\'{c} and Rowlinson \cite{Cvetkovic1988} gave an affirmative answer for the conjecture of Brualdi and Solheid. However, for any positive integers $m$ and $n$,
the Brualdi-Hoffman problem of graphs of order $n$ and size $m$ is still open. In order to determine the extremal graph in the Brualdi-Hoffman problem under the constraint of order, many reports about the analysis of the extremal graph were presented (see, for example, \cite{Andelic2010,Andelic2011,Bhattacharya2008,Petrovic2015,Simic2004,Simic2010}).
Therefore, it is interesting to investigate the Brualdi-Hoffman problem under additional constraints.

Matching theory is a basic subject in graph theory, and it has many important applications in theoretical chemistry and combinatorial optimization \cite{Lovasz1986}. Recall that a matching in a graph is a set of pairwise nonadjacent edges. The number of edges in a maximum matching of a graph $G$ is called the matching number of $G$, and denoted by $\beta(G)$. The matching number of a graph is closed related to the spectral parameters. In \cite{O2010}, O and Cioab\v{a} presented the connections between the eigenvalues and the matching number of a regular graph. Cioab\v{a} and Gregory \cite{Cioaba2007} obtained some spectral sufficient conditions for the existence of large matching in regular graphs. A graph contains a perfect matching if its matching number is half of the order. Some spectral sufficient conditions, which guarantee the graph has a perfect matching, were proved in \cite{Brouwer2005,Cioaba2005,Cioaba2009}.
The matching number and ($Q$-) spectral radius of graphs were investigated in many papers (see \cite{Chang2003,Feng2007a,Feng2007,Hou2002,Lin2007,Li2014,Shen2017,Yu2008}).
Inspired by these observations, we consider the Brualdi-Hoffman problem with the additional matching constrain, that is,
\begin{prob}\label{problem}
  For graphs of size $m$ and matching number $\beta$, what is the maximum of the $Q$-spectral radius?
\end{prob}
We remark that Problem \ref{problem} is trivial when $\beta=1$. Let $G$ be a graph with size $m$ and matching number $\beta=1$. Thus, $G$ is a star or a triangle (only for $m=3$) with possibly some isolated vertices.
In particular, if $m=3$, then the maximum of the $Q$-spectral radius is $q(K_{1,3})=q(K_3)=4$. If $m\neq 3$, then the maximum of the $Q$-spectral radius is $q(K_{1,m})=m+1$.
Let $S_{a,b,c}$ denote the graph obtained from a vertex $v_1$ by attaching $a$ pendant edges, $b$ pendant paths of length $2$
and $c$ pendant triangles (see \textcolor[rgb]{0.00,0.07,1.00}{Fig. \ref{fig1}}).
When $\beta\geq 2$, Problem \ref{problem} is solved in the following theorem.

\begin{thm} \label{th1}
Let $G$ be a graph of matching number $\beta\geq2$ and size $m\geq \beta$.
Then $q(G)\leq q(S_{a,b,c})$,
with equality if and only if $G\cong S_{a,b,c}$ with possibly some isolated edges and isolated vertices.
Moreover,\\
(i) if $m\geq 3\beta-1$, then $a=m-3\beta+3$, $b=0$ and $c=\beta-1$;\\
(ii) if $m\leq 3\beta-2$ and $m-\beta$ is odd, then $a=b=1$ and $c=\frac{m-\beta-1}2$;\\
(iii) if $m\leq 3\beta-2$ and $m-\beta$ is even, then $a=1$, $b=0$ and $c=\frac{m-\beta}2$.
\end{thm}

The proof of Theorem \ref{th1} is provided in the next section. Before proceeding further, let us recall some definitions and notations.
Let $G$ be a graph with signless Laplacian matrix $Q(G)$ and $Q$-spectral radius $q(G)$. It is well-known that
$$q(G)=\max_{||Y||=1}Y^{T}Q(G)Y=X^{T}Q(G)X=\sum_{uv\in E(G)}(x_{u}+x_{v})^{2},$$
where such nonnegative unit eigenvector $X$ is called the principal eigenvector of $Q(G)$.

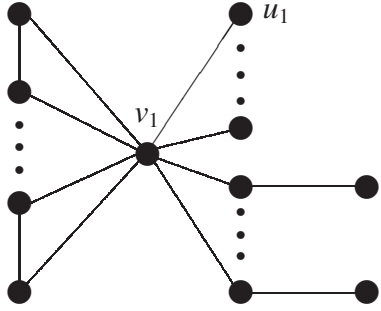
\begin{figure}[t]
\centering \setlength{\unitlength}{2.4pt}
\begin{center}
\unitlength 0.8mm 
\linethickness{0.4pt}
\ifx\plotpoint\undefined\newsavebox{\plotpoint}\fi 
\begin{picture}(59.368,49.236)(0,0)
\put(22.368,24.368){\circle*{4}}
\put(22.368,30.368){\makebox(0,0)[cc]{$v_1$}}
\put(43.868,47.618){\makebox(0,0)[cc]{$u_1$}}
\multiput(22.368,24.368)(-.0337078652,.0389245586){623}{\line(0,1){.0389245586}}
\put(1.118,48.618){\line(0,-1){13.75}}
\multiput(1.118,34.868)(.0666144201,-.0336990596){319}{\line(1,0){.0666144201}}
\put(1.118,47.618){\circle*{4}}
\put(1.118,34.618){\circle*{4}}
\multiput(22.118,24.9)(-.0690789474,-.0317171053){304}{\line(-1,0){.0690789474}}
\put(1.118,16.118){\line(0,-1){13.5}}
\multiput(1.118,1.368)(.0337301587,.0369047619){630}{\line(0,1){.0369047619}}
\put(1.118,16.118){\circle*{4}}
\put(1.118,1.368){\circle*{4}}
\put(1.118,29.118){\circle*{1.2}}
\put(1.118,25.368){\circle*{1.2}}
\put(1.118,21.618){\circle*{1.2}}
\put(22.368,24.618){\line(2,3){15.5}}
\put(37.868,47.618){\circle*{4}}
\put(37.868,41.868){\circle*{1.2}}
\put(37.868,33.118){\circle*{1.2}}
\put(37.868,37.368){\circle*{1.2}}
\multiput(22.368,24.868)(.140625,.033482143){112}{\line(1,0){.140625}}
\put(37.868,28.618){\circle*{4}}
\multiput(22.368,24.118)(.094551282,-.033653846){156}{\line(1,0){.094551282}}
\put(37.868,18.868){\line(1,0){21}}
\put(37.868,18.868){\circle*{4}}
\put(58.868,18.868){\circle*{4}}
\multiput(22.118,25.118)(.0336956522,-.052173913){460}{\line(0,-1){.052173913}}
\put(37.868,1.118){\line(1,0){21}}
\put(58.868,1.368){\circle*{4}}
\put(37.868,1.368){\circle*{4}}
\put(37.868,14.618){\circle*{1.2}}
\put(37.868,10.868){\circle*{1.2}}
\put(37.868,7.368){\circle*{1.2}}
\end{picture}
\end{center}
\caption{\footnotesize{The graph $S_{a,b,c}$, where $a,b,c\geq0$}.}\label{fig1}
\end{figure}

\section{Proofs}

Let $\mathfrak{G}_{m,~\geq\beta}$ be the set of graphs of size $m$ with at least $\beta$ independant edges.
In order to prove \textcolor[rgb]{0.00,0.07,1.00}{Theorem \ref{th1}},
we first look for the extremal graph with maximal $Q$-spectral radius among all graphs in $\mathfrak{G}_{m,~\geq\beta}$.

\begin{lem}\label{le1}
Let $G$ be a non-empty graph and $X$ be the principal eigenvector of $Q(G)$ with coordinate $x_v$ corresponding to $v\in V(G)$.
Assume that $u_1u_2\in E(G)$ and $v_1v_2\notin E(G)$.
If $x_{v_1}+x_{v_2}\geq x_{u_1}+x_{u_2}$ and $x_{v_1}+x_{v_2}>0$, then $q(G-u_1u_2+v_1v_2)>q(G)$.
\end{lem}

\begin{proof}
For convenience, let $G'=G-u_1u_2+v_1v_2$. Then
$$q(G')-q(G)\geq X^TQ(G')X-X^TQ(G)X=(x_{v_1}+x_{v_2})^2-(x_{u_1}+x_{u_2})^2\geq0.$$
Thus, $q(G')\geq q(G).$
If $q(G')=q(G),$
then $q(G')=X^TQ(G')X$ and hence $X$ is also the principal eigenvector of $Q(G')$.
We may assume that $v_2\notin\{u_1,u_2\}$, since $\{v_1,v_2\}\neq \{u_1,u_2\}$.
Then $d_{G'}(v_2)=d_G(v_2)+1$.
Note that $q(G)x_v=d_{G}(v)x_v+\sum_{w\in N_{G}(v)}x_w$ for any $v\in V(G)$.
Thus,
$$q(G')x_{v_2}=d_{G'}(v_2)x_{v_2}+\sum_{w\in N_{G'}(v_2)}x_w=
d_{G}(v_2)x_{v_2}+\sum_{w\in N_{G}(v_2)}x_w+x_{v_1}+x_{v_2}>q(G)x_{v_2},$$
which contradicts $q(G')=q(G)$.
Therefore, $q(G')>q(G)$.
\end{proof}

We now introduce a special matching of a non-empty graph $G$.

\begin{defi}\label{de1}
Let $G$ be a non-empty graph and $X$ be the principal eigenvector of $Q(G)$ with coordinate $x_v$ corresponding to $v\in V(G)$.
A matching $\{u_1v_1, u_2v_2,\ldots, u_{\beta(G)} v_{\beta(G)}\}$ of $G$ is said to be extremal to $X$ and denoted by $M^*(G)$,
if $$\sum_{u_iv_i\in M^*(G)}(x_{u_i}+x_{v_i})^2=\max_M\sum_{uv\in M}(x_u+x_v)^2,$$
where $M$ takes over all the maximum matchings of $G$.
\end{defi}

Let $G^*$ be the extremal graph with maximal $Q$-spectral radius among all graphs in $\mathfrak{G}_{m,~\geq\beta}$,
and $X^*$ be the principal eigenvector of $Q(G^*)$ with coordinate $x^*_v$ corresponding to $v\in V(G^*)$.
The property of the extremal matching of $G^*$ is obtained.

\begin{lem}\label{le2}
Let $M^*(G^*)=\{u_1v_1, u_2v_2,\ldots, u_{\beta(G^*)} v_{\beta(G^*)}\}$ and $V^*=\{u_i,v_i~|~i=1,2,\ldots, \beta(G^*)\}$.
Then $x^*_w\leq \min_{v\in V^*}x^*_v$ for any vertex $w\in V(G^*)\setminus V^*$.
\end{lem}

\begin{proof}
Without loss of generality, we may assume that $x^*_{u_1}=\min_{v\in V^*}x^*_v$.
Let $w$ be an arbitrary vertex in $V(G^*)\setminus V^*$. It suffices to show that $x^*_w\leq x^*_{u_1}$.
Suppose to the contrary that $x^*_w>x^*_{u_1}$.
If $w$ is adjacent to $v_1$ in $G^*$, then
$$\sum_{u_iv_i\in M^*(G)}(x^*_{u_i}+x^*_{v_i})^2<\sum_{uv\in (M^*(G^*)\setminus \{u_1v_1\})\cup \{wv_1\}}(x^*_u+x^*_v)^2,$$
which contradicts the definition of $M^*(G^*)$. Thus, $wv_1\notin E(G^*)$.

Let $G=G^*-u_1v_1+wv_1$. Clearly, $G\in\mathfrak{G}_{m,~\geq\beta}$. Since $X^*$ is a nonnegative vector, we have $x^*_w>x^*_{u_1}\geq0$, and so $x^*_w+x^*_{v_1}>x^*_{u_1}+x^*_{v_1}$.
It follows from \textcolor[rgb]{0.00,0.07,1.00}{Lemma \ref{le1}} that $q(G)>q(G^*)$, which contradicts the maximality of $q(G^*)$.
This completes the proof.
\end{proof}

Let $M^*(G^*)$ and $V^*$ be the sets defined in \textcolor[rgb]{0.00,0.07,1.00}{Lemma \ref{le2}}.
Set $x^*_{v_1}=\max_{v\in V^*}x^*_v$. We define two edge subsets $E_1(G^*)$ and $E_2(G^*)$ of $G^*$,
where $E_1(G^*)=M^*(G^*)\cup \{v_1v~|~v\in N_{G^*}(v_1)\}$ and $E_2(G^*)=E(G^*)\setminus E_1(G^*).$
The following theorem gives a preliminary characterization of the extremal graph $G^*$.

\begin{thm}\label{th2}
If $m(G^*)\geq \beta(G^*)+5$, then $d_{G^*}(u_1)=1$, and $G^*$ is isomorphic to $S_{a,b,c}$
with possibly some isolated edges and isolated vertices.
\end{thm}

\begin{proof}
Recall that $x^*_{v_1}=\max_{v\in V^*}x^*_v$.
By \textcolor[rgb]{0.00,0.07,1.00}{Lemma \ref{le2}},
we have $x^*_{v_1}=\max_{v\in V(G^*)}x^*_v$.
For convenience, let $E_i=E_i(G^*)$ for $i\in\{1,2\}$.
If $E_2=\emptyset$, then the statement holds immediately.
Suppose that $E_2\neq\emptyset$. Let us define two graphs $G_1$ and $G_2$ as follows:

(i) add isolated vertices $w_1,w_2,\ldots,w_{|E_2|}$ to $G^*$
such that $V(G_1)=V(G_2)=V(G^*)\cup\{w_i~|~i=1,2,\ldots,|E_2|\};$

(ii) $E(G_1)=E_1\cup E_2$ and $E(G_2)=E_1\cup E_2'$, where $E_2'=\{v_1w_i~|~i=1,2,\ldots,|E_2|\}.$ \\
Let $X_1=({X^*}^T,0,0,\ldots,0)^T$,
where the number of extended zero-components is $|E_2|$.
Clearly, $q(G_1)=q(G^*)$ and $X_1$ is the principal eigenvector of $Q(G_1)$.
Let $X_2$ be the principal eigenvector of $Q(G_2)$ with coordinate $x_v$ corresponding to $v\in V(G_2)$.
Then
\begin{eqnarray*}
  {X_1}^TX_2[q(G_2)-q(G_1)] &=& {X_1}^Tq(G_2)X_2-[q(G_1)X_1]^TX_2\\
  &=& {X_1}^TQ(G_2)X_2-[Q(G_1)X_1]^TX_2\\
  &=& {X_1}^T[Q(G_2)-Q(G_1)]X_2.
\end{eqnarray*}
Thus,
\begin{eqnarray}\label{eq1}
{X_1}^TX_2[q(G_2)-q(G_1)]=\sum_{v_1w_i\in E_2'}(x^*_{v_1}+0)(x_{v_1}+x_{w_i})-\sum_{uv\in E_2}(x^*_u+x^*_v)(x_u+x_v).
\end{eqnarray}
Note that each $w_i$ is a pendant vertex of $G_2$. Then $q(G_2)x_{w_i}=x_{w_i}+x_{v_1},$
that is,
\begin{eqnarray}\label{eq2}
x_{w_i}=\frac{x_{v_1}}{q(G_2)-1}
\end{eqnarray}
for $i=1,2,\ldots, |E_2|$.
On the other hand, we have
\begin{eqnarray}\label{eq3}
x^*_u+x^*_v\leq 2x^*_{v_1}
\end{eqnarray}
for any $uv\in E_2$, since $x^*_{v_1}=\max_{v\in V(G^*)}x^*_v$.
Moreover, we can see that $d_{G_2}(v)\leq2$ for any $v\in V(G_2)\setminus \{v_1\}$.
Let $x_{v^*}=\max_{v\in V(G_2)\setminus \{v_1\}}x_v$.
Then $$q(G_2)x_{v^*}=d_{G_2}(v^*)x_{v^*}+\sum_{v\in N_{G_2}(v^*)}x_v\leq 3x_{v^*}+x_{v_1},$$
that is, $x_{v^*}\leq\frac{x_{v_1}}{q(G_2)-3}$.
This implies that
\begin{eqnarray}\label{eq4}
x_u+x_v\leq 2x_{v^*}\leq\frac{2x_{v_1}}{q(G_2)-3}
\end{eqnarray}
for any $uv\in E_2$.
Combining (\ref{eq1}), (\ref{eq2}), (\ref{eq3}) and (\ref{eq4}), we have
\begin{eqnarray}\label{eq5}
{X_1}^TX_2[q(G_2)-q(G_1)]\geq[\frac{q(G_2)}{q(G_2)-1}-\frac{4}{q(G_2)-3}]|E_2|x^*_{v_1}x_{v_1}.
\end{eqnarray}
By the definition of $G_2$,  we can see that $m(G_2)=m(G^*)$ and $\beta(G_2)=\beta(G^*)$.
Thus, $G_2$ also belongs to $\mathfrak{G}_{m,~\geq\beta}$.
Furthermore, we have $q(G_2)\leq q(G^*)=q(G_1)$, since $q(G^*)$ is maximal.
Note that $x^*_{v_1}x_{v_1}>0$,
since $x^*_{v_1}=\max_{v\in V(G^*)}x^*_v$
and $v_1$ belongs to the unique connected component of $G_2$ other than an isolated vertex or edge.
By (\ref{eq5}), we have
$\frac{q(G_2)}{q(G_2)-1}\leq\frac{4}{q(G_2)-3}$, i.e., $q^2(G_2)-7q(G_2)+4\leq0$.
Thus,
\begin{eqnarray}\label{eq6}
q(G_2)\leq\frac{7+\sqrt{33}}2<7.
\end{eqnarray}
Recall that $m(G^*)\geq\beta(G^*)+5$, that is, $|E(G^*)\setminus M^*(G^*)|\geq5$.
Note that $u_1v_1\in E(G_2)$ and all the edges of $E(G^*)\setminus M^*(G^*)$ become edges incident to $v_1$ in $G_2$.
It follows that $d_{G_2}(v_1)\geq6$. This implies that $q(G_2)\geq q(K_{1,6})=7$,
which contradicts (\ref{eq6}).
Therefore, $E_2=\emptyset$ and the statement holds.
\end{proof}

In the following, we consider the case $m(G^*)\leq \beta(G^*)+4$.
We now give an ordering of the edges in an extremal matching of $G^*$.

\begin{defi}\label{de2}
An ordering $u_1v_1, u_2v_2, \ldots, u_{\beta(G^*)} v_{\beta(G^*)}$, of the edges in $M^*(G^*)$,
is said to be proper to $X^*$, if it satisfies the following conditions for each $1\leq i\leq \beta(G^*)$:\\
(i) $x^*_{v_i}\geq x^*_{u_i}$;\\
(ii) $x^*_{v_i}\geq x^*_{v_{i+1}}$;\\
(iii) $x^*_{u_i}\geq x^*_{u_{i+1}}$ if $x^*_{v_i}=x^*_{v_{i+1}}$.
\end{defi}

In the following, we may assume that $u_1v_1,u_2v_2,\ldots,u_{\beta(G^*)} v_{\beta(G^*)}$
is a proper ordering of $M^*(G^*)$. Then we have the following results.

\begin{lem}\label{le3}
Let $\beta(G^*)\geq2$ and $i,j\in \{1,2,\ldots,\beta(G^*)\}$ with $i<j$.
Then $x^*_{u_i}\geq x^*_{v_j}$ if and only if $\{u_i,v_i,u_j,v_j\}$ induces two isolated edges or a copy of $K_4$.
\end{lem}

\begin{proof}
Let $H$ be the subgraph of $G^*$ induced by $\{u_i,v_i,u_j,v_j\}$.
Firstly, suppose that $x^*_{u_i}\geq x^*_{v_j}$ and $H\ncong 2 K_2$.
Then there exists an edge $uv\in E(H)\setminus \{u_iv_i,u_jv_j\}$.
Note that $q(G^*)x^*_u=d_{G^*}(u)x^*_u+\sum_{w\in N_{G^*}(u)}x^*_w.$
If $x^*_u=0$, then $x^*_w=0$ for any $w\in N_{G^*}(u)$.
This implies that $v_1\notin N_{G^*}(u)$, since $x^*_{v_1}>0$.
Let $G=G^*-uv+uv_1$. It is obvious that $M^*(G^*)\subseteq E(G)$, and hence $G\in \mathfrak{G}_{m,~\geq\beta}$.
Since $x^*_u+x^*_{v_1}\geq x^*_u+x^*_v$, it follows from \textcolor[rgb]{0.00,0.07,1.00}{Lemma \ref{le1}} that $q(G)>q(G^*)$, a contradiction.
Therefore, $x^*_u>0$.

If $v_iv_j\notin E(H)$, we define  $G'=G^*-uv+v_iv_j$.
Note that $uv\neq u_iv_i$.
By \textcolor[rgb]{0.00,0.07,1.00}{Definition \ref{de2}}, we have $x^*_{v_i}+x^*_{v_j}\geq x^*_u+x^*_v$.
Similarly as above, we have $G'\in \mathfrak{G}_{m,~\geq\beta}$ and $q(G')>q(G^*)$.
Therefore, $v_iv_j\in E(H)$.
Recall that $x^*_{u_i}\geq x^*_{v_j}$. Then $x^*_{u_i}+x^*_{u_j}\geq x^*_{v_j}+x^*_{u_j}$.
\textcolor[rgb]{0.00,0.07,1.00}{Lemma \ref{le1}} implies that $u_iu_j\in E(H)$,
otherwise, $q(G^*-v_ju_j+u_iu_j)>q(G^*)$, and $\beta(G^*-v_ju_j+u_iu_j)\geq\beta(G^*)$
since $(M^*(G^*)\setminus \{u_iv_i,u_jv_j\})\cup \{u_iu_j,v_iv_j\}$ is a matching of $G^*-v_ju_j+u_iu_j$.
Furthermore, note that $x^*_{u_i}+x^*_{v_j}\geq x^*_{u_i}+x^*_{u_j}$ and $x^*_{v_i}+x^*_{u_j}\geq x^*_{u_i}+x^*_{u_j}$.
By \textcolor[rgb]{0.00,0.07,1.00}{Lemma \ref{le1}}, we have $u_iv_j, v_iu_j\in E(H)$.
It follows that $H\cong K_4$.

Conversely, assume that $x^*_{u_i}<x^*_{v_j}$.
We claim that either $u_iu_j\in E(H)$ or $v_iv_j\in E(H)$.
Otherwise, let $G''=G^*-u_iv_i-u_jv_j+u_iu_j+v_iv_j$.
Then $\beta(G'')\geq\beta(G^*)$ and hence $G''\in \mathfrak{G}_{m,~\geq\beta}$.
Moreover,
\begin{eqnarray*}
 q(G'')-q(G^*) &\geq& {X^*}^T[Q(G'')-Q(G^*)]X^*\\
 &=& (x^*_{u_i}+x^*_{u_j})^2+(x^*_{v_i}+x^*_{v_j})^2-(x^*_{u_i}+x^*_{v_i})^2-(x^*_{u_j}+x^*_{v_j})^2\\
 &=& 2(x^*_{u_i}x^*_{u_j}+x^*_{v_i}x^*_{v_j})-2(x^*_{u_i}x^*_{v_i}+x^*_{u_j}x^*_{v_j})\\
 &=& 2(x^*_{v_j}-x^*_{u_i})(x^*_{v_i}-x^*_{u_j}).
\end{eqnarray*}
According to \textcolor[rgb]{0.00,0.07,1.00}{Definition \ref{de2}}, we have $x^*_{v_i}\geq x^*_{v_j}\geq x^*_{u_j}$.
If $x^*_{v_i}=x^*_{u_j}$, then $x^*_{v_i}=x^*_{v_j}$. Now by the ordering rule (iii),
we have $x^*_{u_i}\geq x^*_{u_j}$, a contradiction. Thus, $x^*_{v_i}>x^*_{u_j}$.
Combining with $x^*_{u_i}<x^*_{v_j}$, we have
\begin{eqnarray}\label{eq7}
q(G'')-q(G^*)\geq2(x^*_{v_j}-x^*_{u_i})(x^*_{v_i}-x^*_{u_j})>0,
\end{eqnarray}
a contradiction. Thus, the claim holds and $H\ncong 2 K_2$.

If $H\cong K_4$, we define $M=(M^*(G^*)\setminus \{u_iv_i,u_jv_j\})\cup \{u_iu_j,v_iv_j\}$.
Then $M$ is a maximum matching of $G^*$. However, by the above discussion, $$(x^*_{u_i}+x^*_{u_j})^2+(x^*_{v_i}+x^*_{v_j})^2>(x^*_{u_i}+x^*_{v_i})^2+(x^*_{u_j}+x^*_{v_j})^2,$$
which contradicts that $M^*(G^*)$ is extremal. Hence, $H\ncong K_4$.
This completes the proof.
\end{proof}

\begin{figure}[t]
\centering \setlength{\unitlength}{2.4pt}
\begin{center}
\unitlength 0.8mm 
\linethickness{0.4pt}
\ifx\plotpoint\undefined\newsavebox{\plotpoint}\fi 
\begin{picture}(185.5,53.368)(0,0)
\put(49.5,38.25){\line(1,0){.25}}
\put(0,27){\makebox(0,0)[cc]{}}
\put(185.25,53.368){\line(1,0){.25}}
\put(135.75,42.118){\makebox(0,0)[cc]{}}
\put(182.75,37.868){\line(1,0){.25}}
\put(133.25,26.618){\makebox(0,0)[cc]{}}
\put(151.25,39.368){\line(1,0){.25}}
\put(101.75,28.118){\makebox(0,0)[cc]{}}
\put(116.25,38.868){\line(1,0){.25}}
\put(66.75,27.618){\makebox(0,0)[cc]{}}
\put(30,39.368){\line(0,-1){26}}
\put(62,39.368){\line(0,-1){26.25}}
\multiput(62,39.118)(-.04182879377,-.03372243839){771}{\line(-1,0){.04182879377}}
\put(30.25,39.368){\line(6,-5){31.5}}
\put(62.25,13.118){\circle*{4}}
\put(62,38.618){\circle*{4}}
\put(67.25,39.118){\makebox(0,0)[cc]{$u_2$}}
\put(67.25,13.118){\makebox(0,0)[cc]{$v_2$}}
\put(24.25,39.118){\makebox(0,0)[cc]{$u_1$}}
\put(30,13){\line(1,0){32}}
\put(30,13.118){\circle*{4}}
\put(30,38.618){\circle*{4}}
\put(24.25,13.118){\makebox(0,0)[cc]{$v_1$}}
\put(45.75,0){\makebox(0,0)[cc]{$H_1$}}
\put(97,38.618){\line(0,-1){26}}
\put(129,38.618){\line(0,-1){26.25}}
\put(129,12.368){\line(-1,0){32.25}}
\put(97.25,38.618){\circle*{4}}
\put(97,13.118){\circle*{4}}
\put(129,38.618){\circle*{4}}
\put(91.25,39.118){\makebox(0,0)[cc]{$u_1$}}
\put(91.25,13.118){\makebox(0,0)[cc]{$v_1$}}
\put(128.75,12.25){\circle*{4}}
\put(133.5,39.118){\makebox(0,0)[cc]{$u_2$}}
\put(133.5,13.118){\makebox(0,0)[cc]{$v_2$}}
\put(128.5,12.25){\line(1,0){31.75}}
\put(160,12){\line(0,1){25.5}}
\put(160,38.618){\circle*{4}}
\put(164.25,39.118){\makebox(0,0)[cc]{$u_3$}}
\put(164.25,13.118){\makebox(0,0)[cc]{$v_3$}}
\qbezier(97,12.75)(129.375,0)(160.25,12.25)
\put(129.25,0){\makebox(0,0)[cc]{$H_2$}}
\put(160,13.118){\circle*{4}}
\end{picture}
\end{center}
\caption{\footnotesize{The subgraphs $H_1$ and $H_2$ of $G^*$.}}\label{fig2}
\end{figure}
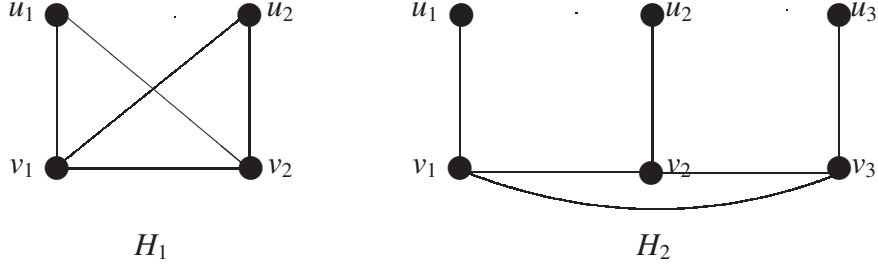

\begin{lem}\label{le4}
If $E_2(G^*)\neq\emptyset$,
then $G^*$ contains either $H_1$ or $H_2$ as a subgraph
(see \textcolor[rgb]{0.00,0.07,1.00}{Fig. \ref{fig2}}).
\end{lem}

\begin{proof}
$E_2(G^*)\neq\emptyset$ implies that $\beta(G^*)\geq2$.
It follows from \textcolor[rgb]{0.00,0.07,1.00}{Definition \ref{de2}} and \textcolor[rgb]{0.00,0.07,1.00}{Lemma \ref{le2}} that,
$x^*_{v_1}=\max_{v\in V(G^*)}x^*_v$ and $x^*_{v_2}=\max_{w\in V(G^*)\setminus \{u_1,v_1\}}x^*_w$.
Let $uv\in E_2(G^*)$. Then $uv\neq u_1v_1$ and hence $x^*_{v_1}+x^*_{v_2}\geq x^*_u+x^*_v$.
Note that $x^*_{v_1}>0$.
If $v_1v_2\notin E(G^*)$, then by \textcolor[rgb]{0.00,0.07,1.00}{Lemma \ref{le1}},
we have $q(G^*-uv+v_1v_2)>q(G^*)$. Note that $G^*-uv+v_1v_2$ also belongs to $\mathfrak{G}_{m,~\geq\beta}$.
Thus, we get a contradiction to the maximality of $q(G^*)$.
Therefore, $v_1v_2\in E(G^*)$.

If $x^*_{u_1}\geq x^*_{v_2}$, then by \textcolor[rgb]{0.00,0.07,1.00}{Lemma \ref{le3}},
$\{u_1,v_1,u_2,v_2\}$ induces a copy of $K_4$. It follows that $H_1$ is a subgraph of $G^*$.

Now suppose that $x^*_{v_2}>x^*_{u_1}.$ Then $x^*_{v_2}=\max_{w\in V(G^*)\setminus \{v_1\}}x^*_w$.
Moreover, $\max\{x^*_{u_1}, x^*_{v_3}\}=\max_{w\in V(G^*)\setminus \{v_1,v_2,u_2\}}x^*_w$.
This indicates that $x^*_{v_2}+\max\{x^*_{u_1}, x^*_{v_3}\}\geq x^*_u+x^*_v$,
since $u\neq v_1, v\neq v_1$ and $uv\neq u_2v_2$.
Thus, \textcolor[rgb]{0.00,0.07,1.00}{Lemma \ref{le1}} implies that either $v_2u_1\in E_2(G^*)$ or $v_2v_3\in E_2(G^*)$.

Notice that $x^*_{v_1}\geq x^*_{v_2}>0$.
If $v_2u_1\in E_2(G^*)$, then $u_2v_1\in E_2(G^*)$
(otherwise, $\beta(G^*-u_2v_2+u_2v_1)\geq \beta(G^*)$ and $q(G^*-u_2v_2+u_2v_1)>q(G^*)$).
Therefore, $H_1\subseteq G^*$.
If $v_2v_3\in E_2(G^*)$, then $v_1v_3\in E_2(G^*)$
(otherwise, $\beta(G^*-v_2v_3+v_1v_3)\geq \beta(G^*)$ and $q(G^*-v_2v_3+v_1v_3)>q(G^*)$).
Therefore, $H_2\subseteq G^*$.
\end{proof}

\begin{thm}\label{th3}
If $m(G^*)\leq \beta(G^*)+4$, then $d_{G^*}(u_1)=1$, and $G^*$ is isomorphic to $S_{a,b,c}$
with possibly some isolated edges and isolated vertices.
\end{thm}

\begin{proof}
Recall that $E_1(G^*)=M^*(G^*)\cup \{v_1v~|~v\in N_{G^*}(v_1)\}$ and $E_2(G^*)=E(G^*)\setminus E_1(G^*).$
It is clear that the statement holds if $E_2(G^*)=\emptyset$.
Now assume that $E_2(G^*)\neq\emptyset$.
Then by \textcolor[rgb]{0.00,0.07,1.00}{Lemma \ref{le4}},
$G^*$ contains either $H_1$ or $H_2$ as a subgraph (see \textcolor[rgb]{0.00,0.07,1.00}{Fig. \ref{fig2}}).
This indicates that $m(G^*)\geq \beta(G^*)+3$.
In particular, if $m(G^*)=\beta(G^*)+3$,
then $G^*$ is isomorphic to either $H_1$ or $H_2$ with possibly some isolated edges and isolated vertices.
A simple calculation shows that $q(H_1)=q(H_2)=3+\sqrt{5}.$
On the other hand, let $G=S_{2,0,1}\cup (\beta(G^*)-2) K_2$.
Clearly, $\beta(G)=\beta(G^*)$ and $m(G)=\beta(G^*)+3=m(G^*)$.
However, $$q(G)=q(S_{2,0,1})\approx 5.3234>3+\sqrt{5}=q(G^*),$$ a contradiction.
Thus, $E_2(G^*)=\emptyset$.

It remains the case $m(G^*)=\beta(G^*)+4$. Let $G'=S_{3,0,1}\cup (\beta(G^*)-2) K_2$.
Clearly, $\beta(G')=\beta(G^*)$ and $m(G')=\beta(G^*)+4=m(G^*)$.
Therefore, $$q(G^*)\geq q(G')=q(S_{3,0,1})>q(K_{1,5})=6,$$
since $K_{1,5}$ is a proper subgraph of $S_{3,0,1}$.
If $\{u_1,v_1,u_2,v_2\}$ induces a $K_4$,
then $G^*$ is isomorphic to $K_4$ with possibly some isolated edges and vertices.
However, $q(G^*)=q(K_4)=6$, which contradicts that $q(G^*)>6$.
Thus, $\{u_1,v_1,u_2,v_2\}$ does not induce a $K_4$.

Firstly, assume that $H_1\subseteq G^*$. Then $u_1u_2\notin E(G^*)$.
Let $uv$ be the unique edge which is not yet determined in $G^*$.
Note that $x^*_{v_1}=\max_{v\in V(G^*)}x^*_v>0$.
Thus by \textcolor[rgb]{0.00,0.07,1.00}{Lemma \ref{le1}},
$uv$ is an edge incident to $v_1$, say, $u=v_1$.
If $v\in V^*$, then $v\in \{u_j,v_j\}$ for some $j\geq3$.
Let $H$ be the subgraph induced by $\{u_1,v_1,u_j,v_j\}$.
Then $H\ncong 2K_2$ since $v_1v\in E(H)\setminus\{u_1v_1,u_jv_j\}$; and $H\ncong K_4$ since $m(G^*)=\beta(G^*)+4$.
By \textcolor[rgb]{0.00,0.07,1.00}{Lemma \ref{le3}}, we have $x^*_{u_1}<x^*_{v_j}.$
Thus $x^*_{v_2}+x^*_{v_j}>x^*_{v_2}+x^*_{u_1}$.
By \textcolor[rgb]{0.00,0.07,1.00}{Lemma \ref{le1}},
we have $q(G^*-v_2u_1+v_2v_j)>q(G^*)$, a contradiction.
Therefore, $v\notin V^*$, that is, $v_1v$ is a pendant edge.
Hence, $G^*$ is isomorphic to $H_3$ with possibly some isolated edges and isolated vertices
(see \textcolor[rgb]{0.00,0.07,1.00}{Fig. \ref{fig3}}).

Secondly, assume that $H_2\subseteq G^*$.
Let $uv$ be the unique edge which is not yet determined in $G^*$.
If $u\in V^*$ and $v\notin V^*$, then by \textcolor[rgb]{0.00,0.07,1.00}{Lemma \ref{le2}},
$x^*_v\leq x^*_{u_2}$ and hence $x^*_{v_1}+x^*_{u_2}\geq x^*_u+x^*_v$.
By \textcolor[rgb]{0.00,0.07,1.00}{Lemma \ref{le1}}, we have $q(G^*-uv+v_1u_2)>q(G^*)$, a contradiction.
Thus, $u,v\in V^*$. If $u\in\{u_j,v_j\}$ for some $j\geq4$, then $v=v_1$,
since $x^*_{v_1}=\max_{w\in V^*}x^*_w$. Now, $\{u_2,v_2,u_j,v_j\}$ induces a copy of $2 K_2$.
By \textcolor[rgb]{0.00,0.07,1.00}{Lemma \ref{le3}}, we conclude that $x^*_{u_2}\geq x^*_u$.
Hence, $x^*_{u_2}+x^*_{v_1}\geq x^*_u+x^*_{v_1}$.
Thus we have $q(G^*-uv_1+u_2v_1)>q(G^*)$, a contradiction.
Therefore, $u,v\in \{u_i,v_i~|~i=1,2,3\}$.
Moreover, $uv\notin \{v_iv_j~|~1\leq i<j\leq3\}$, since $v_1v_2, v_2v_3,v_1v_3\in E(H_2)$.
If $uv\in\{u_iu_j~|~1\leq i<j\leq3\}$, say $uv=u_1u_2$, then $x^*_{v_1}+x^*_{u_2}\geq x^*_u+x^*_v$.
And hence $q(G^*-uv+v_1u_2)>q(G^*)$, a contradiction.
It follows that $u\in\{u_1,u_2,u_3\}$ and $v\in\{v_1,v_2,v_3\}.$
Thus, $G^*$ is isomorphic to $H_4$ with possibly some isolated edges and isolated vertices
(see \textcolor[rgb]{0.00,0.07,1.00}{Fig. \ref{fig3}}).

Notice that $H_3\subseteq H_4$. Thus, in both of above cases, we have
$q(G^*)\leq q(H_4)\approx5.9452,$ which contradicts that $q(G^*)>6$.
Therefore, $E_2(G^*)=\emptyset$ and the statement follows.
\end{proof}

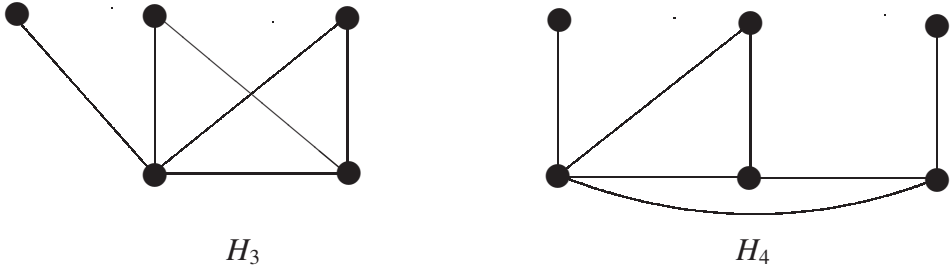
\begin{figure}[t]
\centering \setlength{\unitlength}{2.4pt}
\begin{center}
\unitlength 0.8mm 
\linethickness{0.4pt}
\ifx\plotpoint\undefined\newsavebox{\plotpoint}\fi 
\begin{picture}(185.5,53.368)(0,0)
\put(49.5,38.25){\line(1,0){.25}}
\put(0,27){\makebox(0,0)[cc]{}}
\put(185.25,53.368){\line(1,0){.25}}
\put(135.75,42.118){\makebox(0,0)[cc]{}}
\put(182.75,37.868){\line(1,0){.25}}
\put(133.25,26.618){\makebox(0,0)[cc]{}}
\put(151.25,39.368){\line(1,0){.25}}
\put(101.75,28.118){\makebox(0,0)[cc]{}}
\put(116.25,38.868){\line(1,0){.25}}
\put(66.75,27.618){\makebox(0,0)[cc]{}}
\put(30,39.368){\line(0,-1){26}}
\put(62,39.368){\line(0,-1){26.25}}
\multiput(62,39.118)(-.04182879377,-.03372243839){771}{\line(-1,0){.04182879377}}
\put(30.25,39.368){\line(6,-5){31.5}}
\put(62.25,13.118){\circle*{4}}
\put(62,38.868){\circle*{4}}
\put(30,13){\line(1,0){32}}
\put(30,12.75){\circle*{4}}
\put(30,39.25){\circle*{4}}
\put(97,38.618){\line(0,-1){26}}
\put(129,38.618){\line(0,-1){26.25}}
\put(129,12.368){\line(-1,0){32.25}}
\put(97.25,38.618){\circle*{4}}
\put(97,12.618){\circle*{4}}
\put(129,38.118){\circle*{4}}
\put(128.75,12.25){\circle*{4}}
\put(128.5,12.25){\line(1,0){31.75}}
\put(160,12){\line(0,1){25.5}}
\put(160,37.5){\circle*{4}}
\qbezier(97,12.75)(129.375,0)(160.25,12.25)
\put(160,12){\circle*{4}}
\put(44.75,0){\makebox(0,0)[cc]{$H_3$}}
\put(22.75,40.5){\line(1,0){.25}}
\put(23,40.5){\line(0,1){0}}
\put(23,40.5){\line(0,1){0}}
\put(129.5,0){\makebox(0,0)[cc]{$H_4$}}
\multiput(129,38)(-.04264099037,-.03370013755){727}{\line(-1,0){.04264099037}}
\multiput(29.75,13)(-.03369565217,.03804347826){690}{\line(0,1){.03804347826}}
\put(7.118,39.618){\circle*{4}}
\end{picture}
\end{center}
\caption{\footnotesize{The subgraphs $H_3$ and $H_4$ of $G^*$.}}\label{fig3}
\end{figure}

In the following, we give the proof of \textcolor[rgb]{0.00,0.07,1.00}{Theorem \ref{th1}}.

\medskip

\noindent\textbf{Proof of Theorem \ref{th1}.}
Recall that $G^*$ is the graph with maximal $Q$-spectral radius among all graphs in $\mathfrak{G}_{m,~\geq\beta}$,
where $\beta\geq2$. According to \textcolor[rgb]{0.00,0.07,1.00}{Theorems \ref{th2}-\ref{th3}},
$G^*$ is the disjoint union of $S_{a,b,c}$ with $d$ isolated edges and some isolated vertices,
where $a\geq1$ and $b,c,d\geq0$
(see \textcolor[rgb]{0.00,0.07,1.00}{Fig. \ref{fig1}}).
Let $v_1$ be the vertex of maximal degree in $G^*$ and $u_1v_1$ be a pendant edge.
Clearly, $x_{v_1}=\max_{v\in V(G^*)}x^*_v$ and $\beta(G^*)=b+c+d+1$.

We first claim that $\beta(G^*)=\beta$.
Suppose to the contrary that $\beta(G^*)\geq \beta+1\geq3.$
If $b\geq1$, we define a new graph $G$ by replacing a pendant path of length $2$ with two pendant edges incident to $v_1$.
If $d\geq1$, we define $G$ by replacing an isolated edge with a pendant edge incident to $v_1$.
Note that $a\geq1$. In both cases, we can see $\beta(G)=\beta(G^*)-1\geq \beta$, that is, $G\in\mathfrak{G}_{m,~\geq\beta}$.
However, by \textcolor[rgb]{0.00,0.07,1.00}{Lemma \ref{le1}},
we have $q(G)>q(G^*)$, a contradiction. Thus, $b=d=0$ and hence $\beta(G^*)=c+1$.
This implies that $c\geq2$ and $d_{G^*}(v_1)\geq5$. Hence, $q(G^*)\geq q(K_{1,5})=6$.
Let $\{v_1,u_2,v_2\}$ induce a triangle in $G^*$.
We now define $G$ by adding an isolated vertex $w$ and replacing the edge $u_2v_2$ with a pendant edge $wv_1$.
Then $\beta(G)=\beta(G^*)-1\geq \beta$. Note that $x^*_{u_2}=x^*_{v_2}=\frac{x^*_{v_1}}{q(G^*)-3}\leq \frac{x^*_{v_1}}{3}$.
Thus, $(x^*_{v_1}+0)^2>(x^*_{u_2}+x^*_{v_2})^2$. By \textcolor[rgb]{0.00,0.07,1.00}{Lemma \ref{le1}},
we have $q(G)>q(G^*)$, a contradiction.
The claim holds.

Assume that $a\geq2$.
Then $d=0$, otherwise, we can define a new graph $G'$ by replacing an isolated edge with a pendant edge incident to $u_1$.
Thus, $\beta(G')=\beta(G^*)$ and by \textcolor[rgb]{0.00,0.07,1.00}{Lemma \ref{le1}},
$q(G')>q(G^*),$ a contradiction.
Moreover, $b=0$, otherwise, let $v_1v_2u_2$ be a pendant path of length $2$ and define $G'=G^*-v_2u_2+v_2u_1$.
Observe that $\beta(G')=\beta(G^*)$ and $(q(G^*)-1)x^*_{u_1}=x^*_{v_1}\geq x^*_{v_2}=(q(G^*)-1)x^*_{u_2}$.
We have $x^*_{v_2}+x^*_{u_1}\geq x^*_{v_2}+x^*_{u_2}$.
By \textcolor[rgb]{0.00,0.07,1.00}{Lemma \ref{le1}}, $q(G')>q(G^*)$, a contradiction.
Therefore, $c=\beta-1$ and $a=m-3c=m-3\beta+3$.
This implies that $m\geq 3\beta-1$, since $a\geq2$.

It remains the case $a=1$. Now,
\begin{eqnarray}\label{eq8}
m=2b+3c+d+1.
\end{eqnarray}
Note that $\beta=b+c+d+1$. Thus we have
\begin{eqnarray}\label{eq9}
b+2c=m-\beta.
\end{eqnarray}
If $b\geq2$, say, $v_1v_2u_2$ and $v_1v_3u_3$ are two pendant paths of length $2$.
Define $G''=G^*-u_2v_2-u_3v_3+u_2u_3+v_2v_3$. Then $\beta(G'')=\beta(G^*)$.
Moreover, by symmetry of $G^*$,
we have $x^*_{v_2}=x^*_{v_3}$ and $x^*_{u_2}=x^*_{u_3}=\frac{x^*_{v_2}}{q(G^*)-1}<x^*_{v_2}$.
Similar to \textcolor[rgb]{0.00,0.07,1.00}{Inequality (\ref{eq7})},
we have $$q(G'')-q(G^*)\geq2(x^*_{v_3}-x^*_{u_2})(x^*_{v_2}-x^*_{u_3})>0,$$
a contradiction. Therefore, $b\leq1$.
Combining with \textcolor[rgb]{0.00,0.07,1.00}{Equalities (\ref{eq8})-(\ref{eq9})},
if $m-\beta$ is odd, then $b=1$, $c=\frac{m-\beta-1}2$ and $d=\frac{-m+3\beta-3}2$;
if $m-\beta$ is even, then $b=0$, $c=\frac{m-\beta}2$ and $d=\frac{-m+3\beta-2}2$.
Both cases imply that $m\leq 3\beta-2$.
This completes the proof.\hfill$\Box$



\end{document}